\documentclass[12pt,reqno]{article}

\usepackage{graphicx}
\usepackage{amsmath,amsthm,amsfonts,amssymb,latexsym,mathrsfs,color}
\usepackage{graphicx,cite}
\usepackage{epstopdf}
\usepackage{epsfig}
\usepackage{subfigure}
\usepackage{tikz}
\usetikzlibrary{decorations.pathreplacing,calligraphy}
\usepackage{caption}
\usepackage{enumerate}
\usepackage{caption}
\newtheorem{thm}{Theorem}[section]
\newtheorem{lem}{Lemma}[section]
\newtheorem{coro}{Corollary}[section]
\newtheorem{prop}{Proposition}[section]

\theoremstyle{definition}
\newtheorem{definition}{Definition}[section]
\newtheorem{example}{Example}[section]

\newtheorem{remark}{Remark}[section]
\numberwithin{equation}{section}
\newcommand{\RNum}[1]{\uppercase\expandafter{\romannumeral #1\relax}}

\title{Lorentzian polynomials and log-concavity of  the independence polynomials of  graphs  }
 
\author{Lily Li Liu\footnote{{Corresponding author}
\newline\hspace*{5mm}
{\it Email addresses:}
 liulily@qfnu.edu.cn (L.L. Liu)
 hytang@qfnu.edu.cn (H. Tang),}, Hongyue Tang}
\date{\footnotesize
School of Mathematical Sciences,
Qufu Normal University,
Qufu 273165, PR China}
\begin{document}

\maketitle

\begin{abstract}
In this paper,
we first construct two graphs $\mathcal{F}(l,m,t,s)$ and $\mathcal{G}_4(l,m,t,s)$.
Then we introduce the graph $\mathcal{F}_n(l,m,t,s)$ and the operator $E_{\mathcal{G}_4(l,m,t,s)}$,
where $\mathcal{F}_n(l,m,t,s)$ is defined by identifying the vertex  $c$ of $n$ copies of $\mathcal{F}(l,m,t,s)$,
and $E_{\mathcal{G}_4(l,m,t,s)}$ is defined by replacing each edge of $G$ with $\mathcal{G}_4(l,m,t,s)$,
for any simple finite undirected graph $G$.
By using the theory of Lorentzian polynomials,
we prove that the independence polynomials of the graphs $\mathcal{F}_n(l,m,t,s)$ and the image graphs of $E_{\mathcal{G}_4(l,m,t,s)}$ are log-concave, respectively. 
As applications,
our results not only make progress on the conjecture of Alavi, Malde, Schwenk and Erd\H{o}s, 
but also generalize the results of Bendjeddou and Hardiman.
\bigskip\\
{\sl 2020 MSC:}\quad  05A20; 05C31; 05C69
\bigskip\\
{\sl Keywords:}\quad
Unimodality;
Log-concavity;
Independence polynomial;
Pre-Lorentzian graph;
Lorentzian polynomial
\end{abstract}

\section{Introduction}
\hspace*{\parindent}
An {\it independent set} in a graph is a set of pairwise non-adjacent vertices.
A {\it maximum independent set} of a graph $G$ is a largest independent set and its size is denoted by $\alpha(G)$.
Let $i_k(G)$ denote the number of independent sets of size $k$ in $G$. 
The independence polynomial of $G$ is defined to be
\begin{equation*}
  I(G; x) = \sum^{\alpha(G)}_{k=0}i_{k}(G)x^k,   \quad\quad  i_0(G)=1,
\end{equation*}
(see~\cite{GH1983} for instance).

Let $\alpha=(a_k)^n_{k=0}$ be a sequence of positive numbers. 
The sequence is called {\it unimodal} if there exists an index $0\le m\le n$,
called the {\it mode} of the sequence,
such that $a_0 \leq a_1 \leq \cdots \leq a_{m-1} \leq a_{m} \geq a_{m+1} \geq \cdots \geq a_{n-1} \geq a_{n}$.
The sequence is called {\it log-concave} if $a_{k-1}a_{k+1} \leq a^2_k$ for $k=1, 2,\ldots, n-1$.
Clearly, the log-concavity implies the unimodality.
Unimodal and log-concave sequences occur naturally in combinatorics, algebra, analysis, geometry, probability and statistics.
We refer the reader to the survey papers by Stanley~\cite{Sta89}, Brenti~\cite{Bre89, Bre94} and Br\"and\'en~\cite{Bra15} for various results on the unimodality and log-concavity.
In~\cite{Bre89},
Brenti pointed out that the unimodality and log-concavity have ``one-line" definitions,
but it is very difficult to prove them.
A stronger property is the real-rootedness of its generating function $f(x)=\sum_{k=0}^na_kx^k$ (any root of the polynomial $f(x)$ is real).

There are two motivations to study the log-concavity of independence polynomials.
The first one is the following conjecture,
posed by Alavi, Malde, Schwenk and Erd\H{o}s~\cite{AM1987}:
The independence polynomial of every tree or forest is unimodal.
There have been quite a few results concerned with Alavi et al.'s conjecture (see~\cite{BB2014,B2018,CS07,GH2018,KL2023,KL2025,LM2003,LTZ,LL,TCL,TL25,LM06,LM08,WZ2011,YM,ZZ2023,ZZ2024,Z2016,ZBX18,ZBX22,ZC2019,ZW2020,Z2007}),
which remains open.
Yosef et al.~\cite{YM} proved the log-concavity of independence polynomials of trees with at most  $25$ vertices.
Kadrawi and Levit~\cite{KL2025,KL2023} provided a counterexample,
which showed that Alavi et al.'s conjecture cannot be strengthened to the log-concavity in general.
The other motivation is  based on the result by Chudnovsky and Seymour~\cite{CS07}, 
which states that the independence polynomials of all claw-free graphs are real-rooted. 
Therefore, 
it is natural to consider which clawed graphs also have real-rooted independence polynomials, 
or just have the unimodal and log-concave independence polynomials.

Numerous methods have been developed to prove the unimodality and log-concavity of independence polynomials. 
We present three extremely useful approaches.
The first one is by using the operators of graphs.
Many operators have been applied to prove the real-rootedness of independence polynomials for clawed trees,
such as the clique cover product~\cite{Z2016}, the cycle cover product~\cite{ZBX22} and the incidence product~\cite{ZBX18}  studied by Zhu et al.,
and more recently, the rooted product studied by Zhu et al.~\cite{ZZ2023,ZZ2024} and Liu et al.~\cite{LTZ,TCL}.
The second approach, introduced by Li et al.~\cite{LL}, is based on the theory of symmetric functions.
Li et al.~\cite{LL} obtained that the independence polynomials of all spiders are log-concave.
The third approach is Brändén and Huh's theory of Lorentzian polynomials~\cite{BH2020}.
Using this theory, Brändén and Huh provided an alternative proof of the strongest version of Mason's conjecture, that for any matroid,
the sequence enumerating the independent sets by size is ultra log-concave $\Big($the sequence $\Big(\frac{a_k}{\binom{n}{k}}\Big)_{0\le k\le n}$ is log-concave $\Big)$.
By the theory of Lorentzian polynomials,
Bendjeddou and Hardiman~\cite{BH2025} showed the log-concavity of independence polynomials for graphs obtained by replacing every edge with $W_4$ (see Figure~\ref{f1.3}).
In order to prove the log-concavity,
they introduced the pre-Lorentzian property of graphs,
which is defined via the Lorentzian property of their homogeneous colored independence polynomials multiplied by a suitable monomial.

\begin{figure}[htbp] 
\centering
\begin{tikzpicture}[scale = 0.9]
\draw[line width=1.2pt] (-3,0) -- (-3,-1.5);
\draw[line width=1.2pt] (-1,0) -- (-1,-1.5);
\draw[line width=1.2pt] (1,0) -- (1,-1.5);
\draw[line width=1.2pt] (3,0) -- (3,-1.5);
\draw[line width=1.2pt] (-3,0) -- (-1,0) -- (1,0) -- (3,0);

\fill[red](-3,0) circle(0.6ex);
\fill(-1,0) circle(0.6ex);
\fill(1,0) circle(0.6ex);
\fill[red](3,0) circle(0.6ex);

\fill(-3,-1.5) circle(0.6ex);
\fill(-1,-1.5) circle(0.6ex);
\fill(1,-1.5) circle(0.6ex);
\fill(3,-1.5) circle(0.6ex);

\end{tikzpicture}
\caption{$W_4$
}\label{f1.3}
\end{figure}

In this paper,
following Bendjeddou and Hardiman's idea,
we construct a family of pre-Lorentzian graphs $\mathcal{F}_n(l,m,t,s)$.
Using their pre-Lorentzian property, 
we prove that for any graph formed by replacing each edge with the graph $\mathcal{G}_4(l,m,t,s)$,
which generalizes $W_4$,
the independence polynomial is log-concave.
The graphs $\mathcal{F}_n(l,m,t,s)$ and the image graphs of $E_{\mathcal{G}_4(l,m,t,s)}$ not only contain claws,
but are  also trees in some special cases.
Our results provide further evidence for the conjecture of Alavi, Malde, Schwenk and Erd\H{o}s, 
and also generalize the results of Bendjeddou and Hardiman.

The organization of this paper is as follows.
In section~2,
we review some basic facts about the Lorentzian polynomial, the pre-Lorentzian graph and certain straightforward lemmas from linear algebra.
In section~3,
we first introduce the pre-Lorentzian graph $\mathcal{F}_n(l,m,t,s)$ and the operator $E_{\mathcal{G}_4(l,m,t,s)}$.
Then, for every simple finite undirected graph $G$,
 we prove that the independence polynomials of the image graphs of $E_{\mathcal{G}_4(l,m,t,s)}$ are log-concave.

\section{Preliminaries}
\hspace*{\parindent}
In this section,
we review some basic facts,
which can also be found in~\cite{BH2025} and~\cite{BH2020}.

\subsection{Lorentzian polynomials}
\hspace*{\parindent}
Let $\mathbb{N},\ \mathbb{N}^*,\ \mathbb{Z}$ and $\mathbb{R}$ denote the set of nonnegative integers, positive integers, integers and real numbers, respectively.
For $n \in \mathbb{N}$, we use $[n]$ to denote the set $\{1,2,\ldots,n \}$. 
For all $\alpha \in \mathbb{N}^n $, denote $x^{\alpha}$ by $\prod{x_i^{\alpha_i}}$. 
Assume that the coefficients of all polynomials are over $\mathbb{R}$. 
For $n \in \mathbb{N}^*$, we use $<$ and $\leq$ to denote the partial orders given by
\begin{align*}
 \alpha < \beta \Longleftrightarrow  \alpha_i < \beta_i   \quad \mathrm{and} \quad \alpha \leq \beta \Longleftrightarrow \alpha_i \leq \beta_i , \quad \mathrm{for}\ i \in [n].
\end{align*}

Let $e_i$ denote the standard basis vector $(0, \ldots, 0, 1, 0, \ldots, 0)$, where $1$ is in the $i$th position for $1\le i\le n$.
Following Br\"and\'en and Huh~\cite{BH2020},
we recall the definitions and properties of the Lorentzian polynomials.

\begin{definition}\label{d1.5}
 A finite subset $A \subseteq \mathbb{Z}^n$ is called {\it $M$-convex} if for any $\alpha, \beta \in A$ and any index $i$
 satisfying $\alpha_i > \beta_i$, there is an index $j$ satisfying
 \begin{align*}
  \alpha_j<\beta_j\ \mathrm{and} \  \alpha-e_i+e_j \in A.
 \end{align*}
\end{definition}

\begin{definition}\label{d1.6}
The {\it support} of a multivariate polynomial $p = \sum_{\alpha \in \mathbb{N}^n}c_\alpha x^\alpha $ is given by the set $$\mathrm{Supp}(p) = \{\alpha \in \mathbb{N}^n \colon c_\alpha \neq 0 \}.$$
\end{definition}

\begin{definition}\label{d1.7}
A homogeneous polynomial $p(x_1,x_2,\ldots,x_n)$ of degree $d$ with nonnegative coefficients is called
{\it Lorentzian} if
 \begin{itemize}
   \item Supp$(p)$ is $M$-convex,
   \item for every $(d - 2)$th partial derivative of $p$, the corresponding Hessian matrix has at most one positive eigenvalue, where the Hessian matrix of $p$ is the symmetric matrix $$\mathcal{H}_p(x)=(\partial_i\partial_j p)_{i,j=1}^n.$$
 \end{itemize}

\end{definition}

It is well known that the univariate monomial $ax^n$ and the multivariate linear polynomial in $n$ variables $a_1x_1+a_2x_2+\cdots+a_nx_n$ are Lorentzian polynomials, respectively.
The Lorentzian polynomials have the following two properties.

\begin{prop}[\rm\cite{BH2020}]\label{p1.1}
A Lorentzian polynomial is also Lorentzian under identifying variables.
\end{prop}

\begin{prop}[\rm\cite{BH2020}]\label{p1.2}
The product of two Lorentzian polynomials is a Lorentzian polynomial.
\end{prop}

\subsection{Pre-Lorentzian graphs}
\hspace*{\parindent}
The notion of pre-Lorentzian graph was introduced by Bendjeddou and Hardiman~\cite{BH2025}.

\begin{definition}\label{d1.2}
 A {\it coloured graph} $\mathcal{G}$ is a pair $(G, i)$, where $G$ is a graph and $i$ is a map from $V(\mathcal{G})$ to some indexing set $I$, called {\it the set of colours}.
\end{definition}

\begin{definition}[Free vertex]\label{d1.3}
For a vertex $v$ in a coloured graph, we say that $v$ (or its colour) is
{\it free} if the subset of vertices that share a colour with $v$ is $\{v\}$.
\end{definition}

\begin{definition}[Coloured independence polynomial]\label{d1.4}
Let $\mathcal{G} = (G, i)$ be a finite coloured graph.
The {\it coloured independence polynomial}, denoted as $C(\mathcal{G})$, is defined by 
\begin{align*}
  C(\mathcal{G}; x_{i(v)})=\sum_{S\subseteq V(G)}x_S\Big|_{x_v=x_{i(v)}},
\end{align*}
where $S$ is an independent set of $G$ and $x_S = \prod_{v \in S}x_{v}$.
\end{definition}

So the coloured independence polynomial is obtained by identifying variables corresponding to vertices which share the same colour.
We also say that a variable $x_i$ is {\it free} if the associated colour is.
\begin{definition}\label{d1.8}
A {\it partitioned graph} is a coloured graph with a distinguished colour called the
{\it bound colour} such that every colour other than the bound colour is free.
\end{definition}
As before, we extend the bound colour to vertices and variables in the obvious way.

\begin{definition}\label{d1.9}
Let $p$ be a polynomial in $x_1,x_2,\ldots,x_n$. We write $H(p)$ to denote the homogeneous polynomial in $x_1,x_2,\ldots,x_n, y$ of minimal degree which satisfies $H(p)|_{y=1}=p$. We call $H(p)$ the {\it homogenisation} of $p$ and $y$ the {\it homogenising variable}.
\end{definition}

\begin{definition}\label{d1.10}
Let $\mathcal{G}$ be a partitioned graph. We call $\mathcal{G}$ {\it pre-Lorentzian} if there exists a $k \in \mathbb{N}$ such
that $$(xy)^k H(C(\mathcal{G}))$$ is a Lorentzian polynomial, where $x$ is the variable associated with the bound colour and $y$ is the homogenising variable.
\end{definition}

By the definition of Lorentzian polynomials,
it would therefore be equivalent to
say that $\mathcal{G}$ is pre-Lorentzian if
\begin{itemize}
  \item $\mathrm{Supp} (H(C(\mathcal{G})))$ is $M$-convex,
  \item for every $(n + 2k - 2)$th partial derivative of $(xy)^k \cdot H(C(\mathcal{G}))$, where $n$ is the degree of $C(\mathcal{G})$, the corresponding Hessian matrix has at most one positive eigenvalue.
\end{itemize}

In this paper,
we consider two gluing processes.
The first one is the identification of vertices,
which is defined by identifying one vertex of $n$ graphs to a common new vertex.
This operator is often used in the graph theory to construct more complex graphs from simpler ones.
For example,
the star $K_{1,n}$ can be obtained by identifying one endpoint of $n$ copies of the path $P_2$.
The other one is introduced by Bendjeddou and Hardiman~\cite{BH2025}.
\begin{definition}
Given two finite coloured graphs $\mathcal{G}_1=(G_1, i_1)$ and $\mathcal{G}_2=(G_2, i_2)$ with colour sets $I_1$ and $I_2$,
we may assume, up to changing the colours, 
that
$I_1 \cap I_2 =  \emptyset $. 
The procedure also takes as input two colours $c_1, c_2 \in \mathrm{Im} i_1 \sqcup \mathrm{Im} i_2$. The resulting glued graph, denoted as $Gl(\mathcal{G}_1, \mathcal{G}_2; c_1, c_2)$, is constructed as follows:
\begin{enumerate}
  \item Let $I_3$ be the disjoint union of $I_1$ and $I_2$ with $c_1$ and $c_2$ identified into a single colour which, by convention, we call $c_1$. We consider the disjoint union of $\mathcal{G}_1$ and $\mathcal{G}_2$, with the induced colouring on $I_3$.
  \item We then add any missing edges to the subgraph induced by $i^{-1}(c_1) \subseteq V({G}_1) \sqcup V({G}_2)$ until it forms a clique. The result graph is $Gl(\mathcal{G}_1, \mathcal{G}_2; c_1, c_2)$.
\end{enumerate}
\end{definition}

To indicate the difference from the previous identification, we refer to this gluing procedure as the N-gluing for short.
We give an example to explain the N-gluing.

\begin{example}
We consider the following two coloured graphs $\mathcal{G}_1$ (see Figure \ref{f2}) and $\mathcal{G}_2$ (see Figure \ref{f3}).
\vspace{1cm}

\begin{figure}[htbp]
\centering
\begin{minipage}[h]{0.48\textwidth}
\centering
\begin{tikzpicture}[scale = 1]
\draw[line width=1.2pt] (5,0) -- (5,1) -- (5,-1);
\draw[line width=1.2pt] (5,0) -- (4,0.5);
\draw[line width=1.2pt] (5,1) -- (4,0.5);
\draw[line width=1.2pt] (5,-1) -- (4,-0.5);
\draw[line width=1.2pt] (3.3,0) -- (4,-0.5);
\fill[blue](5,0) circle(0.6ex);
\fill[blue](5,1) circle(0.6ex);
\fill[blue](5,-1) circle(0.6ex);
\fill[brown](4,0.5) circle(0.6ex);
\fill[red](4,-0.5) circle(0.6ex);
\fill[red](3.3,0) circle(0.6ex);
\end{tikzpicture}
\caption{$\mathcal{G}_1$}\label{f2}
\end{minipage}
\begin{minipage}[h]{0.48\textwidth}
\centering
\begin{tikzpicture}[scale = 1]
\draw[line width=1.2pt] (5,0) -- (5,1) -- (5,-1);
\draw[line width=1.2pt] (5,0) -- (4,0);
\draw[line width=1.2pt] (5,-1) -- (4,-1);
\draw[line width=1.2pt] (5,0) -- (6,-0.5);
\draw[line width=1.2pt] (5,-1) -- (6,-0.5);

\fill(5,0) circle(0.6ex);
\fill(5,1) circle(0.6ex);
\fill(5,-1) circle(0.6ex);
\fill[green](4,0) circle(0.6ex);
\fill[green](4,-1) circle(0.6ex);
\fill[orange](6,-0.5) circle(0.6ex);
\end{tikzpicture}
\caption{$\mathcal{G}_2$}\label{f3}
\end{minipage}
\end{figure}

Then the coloured graph, obtained from the N-gluing with respect to the colours blue and green, is
given by Figure \ref{f4},
\begin{figure}[htbp]
\centering

\begin{tikzpicture}[scale = 1]
\draw[magenta, line width=1.2pt] (2,0) -- (4,0);
\draw[magenta, line width=1.2pt] (2,0) -- (4,-1);
\draw[magenta, line width=1.2pt] (1.5,1) -- (4,0);
\draw[magenta, line width=1.2pt] (1.5,1) -- (4,-1);
\draw[magenta, line width=1.2pt] (2,-1) -- (4,0);
\draw[magenta, line width=1.2pt] (2,-1) -- (4,-1);
\draw[magenta, line width=1.2pt] (4,0) -- (4,-1);
\draw[magenta,line width=1.2pt] (2,-1)  -- (1.5,1);

\draw[line width=1.2pt] (2,0)  -- (1.5,1);
\draw[line width=1.2pt] (2,0)  -- (2,-1);
\draw[line width=1.2pt] (2,0) -- (1,0);
\draw[line width=1.2pt] (1.5,1) -- (1,0);
\draw[line width=1.2pt] (2,-1) -- (1,-0.5);
\draw[line width=1.2pt] (0.3,0) -- (1,-0.5);
\fill[blue](2,0) circle(0.6ex);
\fill[blue](1.5,1) circle(0.6ex);
\fill[blue](2,-1) circle(0.6ex);
\fill[brown](1,0) circle(0.6ex);
\fill[red](1,-0.5) circle(0.6ex);
\fill[red](0.3,0) circle(0.6ex);

\draw[line width=1.2pt] (5,0) -- (5,1) -- (5,-1);
\draw[line width=1.2pt] (5,0) -- (4,0);
\draw[line width=1.2pt] (5,-1) -- (4,-1);
\draw[line width=1.2pt] (5,0) -- (6,-0.5);
\draw[line width=1.2pt] (5,-1) -- (6,-0.5);
\fill(5,0) circle(0.6ex);
\fill(5,1) circle(0.6ex);
\fill(5,-1) circle(0.6ex);
\fill[blue](4,0) circle(0.6ex);
\fill[blue](4,-1) circle(0.6ex);
\fill[orange](6,-0.5) circle(0.6ex);

\end{tikzpicture}
\caption{$Gl(\mathcal{G}_1, \mathcal{G}_2; \mathrm{blue}, \mathrm{green})$}\label{f4}
\end{figure}
where the added edges have been drawn in magenta.
\end{example}

The following two properties of the pre-Lorentzian graph play an important role in our proofs.

\begin{lem}[\rm\cite{BH2025}]\label{l1.1}
If the graph is obtained by N-gluing two Pre-Lorentzian graphs across the free vertices, 
then it is also Pre-Lorentzian.
\end{lem}

\begin{lem}[\rm\cite{BH2025}]\label{l1.4}
If a partitioned graph is pre-Lorentzian, then its independence polynomial is log-concave.
\end{lem}

\subsection{Lemmas}
\hspace*{\parindent}
Let $G=(V,E)$ be a simple graph and $v\in V$.
Denote by $N(v)=\{w\in V:vw\in E\}$ and $N[v]=N(v)\cup \{v\}$.
Note that the independence polynomials satisfy the following recurrence relations,
which will be used in calculating the coloured independence polynomials.
\begin{lem}[\cite{GH1983}]\label{l2.1}
Let $G=(V, E)$ be a simple graph. Then
\begin{enumerate}[(i)]
  \item $I(G; x) = I(G-v; x) + x I(G - N[v]; x)$, for any $v\in V(G)$;
  \item $I(G; x) = I(G-e; x) - x^{2}I(G - N(u) \cup N(v); x)$, for any $e=uv\in E(G)$.
\end{enumerate}
\end{lem}

We use the following two lemmas to study the eigenvalues of the Hessian matrices.
Their proofs are based on the properties of matrices from linear algebra~\cite{BH2025}.

\begin{lem}[\rm\cite{BH2025}]\label{l1.2}
Let $M$ be a $(n+2)\times(n+2)$ symmetric matrix of the form

$$
M=\begin{pmatrix}
0        & a        & \cdots   & a        & b       & c        \\
a        & 0        & \ddots   & \vdots   & \vdots  & \vdots   \\
\vdots   & \ddots   & \ddots   & a        & b       & c        \\
a        & \cdots   & a        & 0        & b       & c        \\
b        & \cdots   & b        & b        & d       & f        \\
c        & \cdots   & c        & c        & f       & e        \\
\end{pmatrix},
$$
where $a$, $b$, $c$, $d$, $e$, $f$ are nonnegative real numbers. Then $M$ has $n-1$ nonpositive eigenvalues of the form $-a$ and the corresponding eigenspace is spanned by the vectors $(e_1 - e_i)_{2\leq i \leq n}$. Furthermore, any of the remaining three eigenvalues is either equal to $-a$ or is an eigenvalue of the matrix
$$
B =\begin{pmatrix}
(n-1)a      & b        & c     \\
nb          & d        & f      \\
nc          & f        & e     \\
\end{pmatrix}.
$$

\end{lem}
Within the context of this lemma, we call $B$ {\it the reduced form} of $M$.

\begin{lem}[\rm\cite{BH2025}]\label{l1.3}
Let $A$ be a $3 \times 3$ matrix which satisfies
\begin{equation}\label{e1.1}
  \mathrm{Tr}(A) > 0, A_{11}+ A_{22} + A_{33} \leq 0 \  and \ \mathrm{det}(A) \geq 0,
\end{equation}
where $A_{ii}$ is the minor of $A$ obtained by deleting  the $i$th row and the $i$th column. Then $A$ has exactly one positive eigenvalue.
\end{lem}

\section{Main results}
\hspace*{\parindent}
In this section,
we first define the coloured graph $\mathcal{F}(l,m,t,s)$ with only one free vertex.
And we construct the graph $\mathcal{F}_n(l,m,t,s)$ by identifying the center $c$ of $n$ copies of $\mathcal{F}(l,m,t,s)$.
Then we show that the graph $\mathcal{F}_n(l,m,t,s)$ is a pre-Lorentzian graph.
After N-gluing two copies of $\mathcal{F}(l,m,t,s)$,
we introduce the graph $\mathcal{G}_4(l,m,t,s)$.
Finally,
we define the operator $E_{\mathcal{G}_4(l,m,t,s)}$
and show that the independence polynomial of the image of $E_{\mathcal{G}_4(l,m,t,s)}$ is log-concave
by using the pre-Lorentzian property of $\mathcal{F}_n(l,m,t,s)$.
We always assume that $m,s,t,l \in \mathbb{N}$, $l\geq m \geq s$ and $l \geq t+m$ in the sequel.

\begin{definition}\label{d1.2.1}
Given the vertex $c$, the coloured vertex $x_i$ and four complete graphs $K_m$, $K_l$, $K_{m-s}$ and $K_{l-t}$,
define the coloured graph $\mathcal{F}(l,m,t,s)$ by adding $s$, $t$, $m-s$ and $l-t-m$ edges connecting the vertex $c$ with the complete graphs $K_m$, $K_l$, $K_{m-s}$, $K_{l-t}$, respectively,
and by adding $m$ and $l-m$ edges connecting the vertex $x_i $ with the complete graphs $K_m$, $K_l$, respectively
(see Figure~\ref{f1.2.1}).
\begin{figure}[htbp] 
\centering
\begin{tikzpicture}[scale = 0.8]
\draw[line width=1.2pt](-3,0) ellipse (2cm and 1cm) node {$K_{m-s}$};
\draw[line width=1.2pt](3,0) ellipse (2cm and 1cm)node {$K_{m}$};
\draw[line width=1.2pt](-3,-6) ellipse (2cm and 1cm)node {$K_{l-t}$};
\draw[line width=1.2pt](3,-6) ellipse (2cm and 1cm)node {$K_{l}$};

\fill(-3,-3) circle(0.6ex)node [left=3pt] {$c$};
\fill[blue](3,-3) circle(0.6ex)node [right=3pt] {$x_i$};
\fill(-3,-1) circle(0.6ex);
\fill(3,-1) circle(0.6ex);
\fill(-3,-5) circle(0.6ex);
\fill(3,-5) circle(0.6ex);
\fill(1,0) circle(0.6ex);
\fill(1,-6) circle(0.6ex);

\draw[line width=1.2pt] (-3,-1) -- (-3,-3)node[midway, left]{$m-s$};
\draw[line width=1.2pt] (-3,-5) -- (-3,-3)node[midway, left]{$l-t-m$};
\draw[line width=1.2pt] (3,-3) -- (-3,-3);
\draw[line width=1.2pt] (3,-3) -- (3,-1)node[midway, right]{$m$};
\draw[line width=1.2pt] (3,-3) -- (3,-5)node[midway, right]{$l-m$};
\draw[line width=1.2pt] (-3,-3) -- (1,0)node[midway, above]{$s$};
\draw[line width=1.2pt] (-3,-3) -- (1,-6)node[midway, below]{$t$};



\end{tikzpicture}
\caption{$\mathcal{F}(l,m,t,s)$
}\label{f1.2.1}
\end{figure}
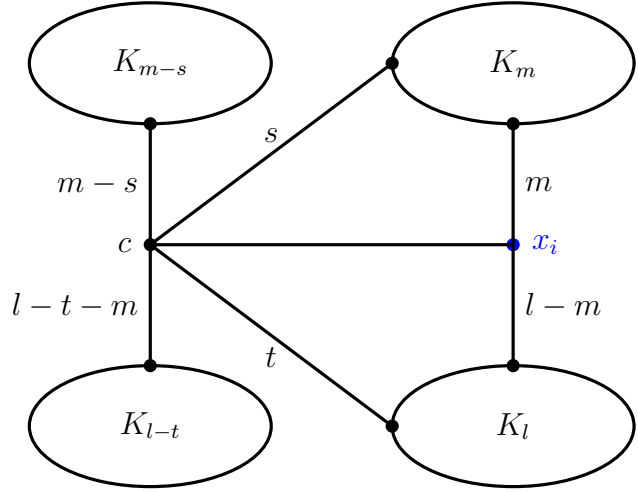
\end{definition}

 \begin{remark}
We note that the added edges may be chosen arbitrarily among the vertices of the corresponding complete graphs in Definition~\ref{d1.2.1}.
All results hold for every such choice.
 \end{remark}
 
 In particular,
 $\mathcal{F}(1,0,0,0)$ and $\mathcal{F}(2,1,0,0)$ are trees (see Figures~\ref{f6-1} and~\ref{f6-2}, respectively).

\begin{figure}
\begin{center}
\begin{minipage}[h]{0.48\textwidth}
\centering
\setlength{\unitlength}{1cm}
\centering
\begin{tikzpicture}[scale = 1]
\draw[line width=1.2pt] (-4,0) -- (-1,0);
\draw[line width=1.2pt] (-4,0) -- (-4,-1.5);
\draw[line width=1.2pt] (-1,-1.5) -- (-1,0);

\fill(-4,0) circle(0.6ex)node [left=3pt] {$c$};
\fill[blue](-1,0) circle(0.6ex)node [right=3pt] {$x_i$};
\fill(-4,-1.5) circle(0.6ex);
\fill(-1,-1.5) circle(0.6ex);

\end{tikzpicture}
\caption{$\mathcal{F}(1,0,0,0)$}\label{f6-1}
\end{minipage}
\begin{minipage}[h]{0.48\textwidth}
\centering
\begin{tikzpicture}[scale = 1]
\draw[line width=1.2pt] (1,0) -- (4,0);
\draw[line width=1.2pt] (4,0) -- (4,-1)--(4,-2)--(4,1);
\draw[line width=1.2pt] (1,-1) -- (1,0)--(1,-2)--(1,1);

\fill(1,0) circle(0.6ex)node [left=3pt] {$c$};
\fill[blue](4,0) circle(0.6ex)node [right=3pt] {$x_i$};
\fill(4,-1) circle(0.6ex);
\fill(1,-1) circle(0.6ex);
\fill(4,-2) circle(0.6ex);
\fill(1,-2) circle(0.6ex);
\fill(4,1) circle(0.6ex);
\fill(1,1) circle(0.6ex);

\end{tikzpicture}
\caption{$\mathcal{F}(2,1,0,0)$}\label{f6-2}
\end{minipage}
\end{center}
\end{figure}

\begin{definition}\label{d1.2.2}
For any $n \in \mathbb{N}^*$, 
given the $n$ copies of $\mathcal{F}(l,m,t,s)$ with distinct free vertices $x_1,x_2,\ldots,x_n$,
define the coloured graph $\mathcal{F}_n(l,m,t,s)$ by identifying the vertex $c$ of $n$ copies of $\mathcal{F}(l,m,t,s)$ (see Figure~\ref{f1000} for $\mathcal{F}_n(1,0,0,0)$ and Figure~\ref{f6} for $\mathcal{F}_n(2,1,0,0)$ as examples).

\begin{figure}[htbp]
\centering
\begin{minipage}[h]{0.48\textwidth}
\centering
\begin{tikzpicture}[scale = 1]

\draw[line width=1.2pt] (-3.5,0) -- (-5,0);
\draw[line width=1.2pt] (-3.5,0) -- (-5,1);
\draw[line width=1.2pt] (-3.5,0) -- (-5,-2);
\draw[line width=1.2pt] (-3.5,0) -- (-2,0);
\draw[line width=1.2pt] (-3.5,0) -- (-2,1);
\draw[line width=1.2pt] (-3.5,0) -- (-2,-2);
\draw[line width=1.2pt] (-0.5,0) -- (-2,0);
\draw[line width=1.2pt] (-0.5,1) -- (-2,1);
\draw[line width=1.2pt] (-0.5,-2) -- (-2,-2);

\fill(-3.5,0) circle(0.6ex) node [above=5pt] {$c$};
\fill(-5,0) circle(0.6ex);
\fill(-5,1) circle(0.6ex);
\fill(-5,-2) circle(0.6ex);
\fill(-5,-0.5) circle(0.3ex);
\fill(-5,-1) circle(0.3ex);
\fill(-5,-1.5) circle(0.3ex);
\fill[red](-2,1) circle(0.6ex) node [above=5pt,right=3pt] {$x_1$};
\fill[green](-2,0) circle(0.6ex) node [above=5pt,right=3pt] {$x_2$};
\fill[blue](-2,-2) circle(0.6ex) node [below=5pt,right=3pt] {$x_n$};
\fill(-2,-0.5) circle(0.3ex);
\fill(-2,-1) circle(0.3ex);
\fill(-2,-1.5) circle(0.3ex);
\fill(-0.5,0) circle(0.6ex);
\fill(-0.5,1) circle(0.6ex);
\fill(-0.5,-2) circle(0.6ex);
\end{tikzpicture}
\caption{$ \mathcal{F}_n(1,0,0,0)$}\label{f1000}
\end{minipage}
\begin{minipage}[h]{0.48\textwidth}
\centering
\begin{tikzpicture}[scale = 1]

\draw[line width=1.2pt] (0.5,0) -- (-1,0);
\draw[line width=1.2pt] (0.5,0) -- (-1,1);
\draw[line width=1.2pt] (0.5,0) -- (-1,-2);
\draw[line width=1.2pt] (0.5,0) -- (-1,0.5);
\draw[line width=1.2pt] (0.5,0) -- (-1,1.5);
\draw[line width=1.2pt] (0.5,0) -- (-1,-2.5);

\draw[line width=1.2pt] (-2,0) -- (-1,0);
\draw[line width=1.2pt] (-2,1) -- (-1,1);

\draw[line width=1.2pt] (-2,-2.5) -- (-1,-2.5);

\draw[line width=1.2pt] (0.5,0) -- (2,0);
\draw[line width=1.2pt] (0.5,0) -- (2,1);
\draw[line width=1.2pt] (0.5,0) -- (2,-2);

\draw[line width=1.2pt] (3.5,0) -- (2,0);
\draw[line width=1.2pt] (3.5,1) -- (2,1);
\draw[line width=1.2pt] (3.5,-2) -- (2,-2);
\draw[line width=1.2pt] (3.5,0.5) -- (2,0);
\draw[line width=1.2pt] (3.5,1.5) -- (2,1);
\draw[line width=1.2pt] (3.5,-1.5) -- (2,-2);

\draw[line width=1.2pt] (3.5,0) -- (4.5,0);
\draw[line width=1.2pt] (3.5,1) -- (4.5,1);
\draw[line width=1.2pt] (3.5,-2) -- (4.5,-2);

\fill(0.5,0) circle(0.6ex) node [above=5pt] {$c$};

\fill(-1,0) circle(0.6ex);
\fill(-1,1) circle(0.6ex);
\fill(-1,-2) circle(0.6ex);
\fill(-1,0.5) circle(0.6ex);
\fill(-1,1.5) circle(0.6ex);
\fill(-1,-2.5) circle(0.6ex);

\fill(-1.5,-0.5) circle(0.3ex);
\fill(-1.5,-1) circle(0.3ex);
\fill(-1.5,-1.5) circle(0.3ex);

\fill(-2,0) circle(0.6ex);
\fill(-2,1) circle(0.6ex);

\fill(-2,-2.5) circle(0.6ex);

\fill[red](2,1) circle(0.6ex) node [above=10pt,right=3pt] {$x_1$};
\fill[green](2,0) circle(0.6ex) node [above=10pt,right=3pt] {$x_2$};
\fill[blue](2,-2) circle(0.6ex) node [below=5pt,right=3pt] {$x_n$};

\fill(2,-0.5) circle(0.3ex);
\fill(2,-1) circle(0.3ex);
\fill(2,-1.5) circle(0.3ex);

\fill(3.5,0) circle(0.6ex);
\fill(3.5,1) circle(0.6ex);
\fill(3.5,-2) circle(0.6ex);
\fill(3.5,0.5) circle(0.6ex);
\fill(3.5,1.5) circle(0.6ex);
\fill(3.5,-1.5) circle(0.6ex);

\fill(4.5,0) circle(0.6ex);
\fill(4.5,1) circle(0.6ex);
\fill(4.5,-2) circle(0.6ex);

\end{tikzpicture}
\caption{$ \mathcal{F}_n(2,1,0,0)$}\label{f6}
\end{minipage}
\end{figure}

\end{definition}

The bound vertices have been drawn in black and the vertex $c$ is called {\it the center}.

In order to prove the pre-Lorentzian property of $\mathcal{F}_n(l,m,t,s)$,
we need to show the Lorentzian property of the following multivariate polynomials $f(x_1,x_2,x,y)$.
\begin{lem}\label{p3.2}
Suppose that $a, b \in \mathbb{N}^*$ and $$f(x_1,x_2,x,y) = (ax + y)[(bx + x_1 + y)(bx + x_2 + y) + xy].$$ If $a=1$ or $a=b$, then 
 $f(x_1,x_2,x,y)$ is a Lorentzian polynomial. 
\end{lem}

\begin{proof}
 For any $a, b \in \mathbb{N}^*$, 
 the support
\begin{align*} \mathrm{Supp}(f(x_1,x_2,x,y))&= \mathrm{Supp}((ax + y)[(bx + x_1 + y)(bx + x_2 + y) + xy])\\
&=\mathrm{Supp}( (ax + y)(bx + x_1 + y)(bx + x_2 + y)+ax^2y+axy^2)\\
&=\mathrm{Supp}( (ax + y)(bx + x_1 + y)(bx + x_2 + y)).
\end{align*}
By Proposition~\ref{p1.2} and the Lorentzian property of the multivariate linear polynomials,
we obtain that $(ax + y)(bx + x_1 + y)(bx + x_2 + y)$ is a Lorentzian polynomial.
Thus $\mathrm{Supp}(f(x_1,x_2,x,y))$ is M-convex. 

Note that $$\partial_{x_1}f(x_1,x_2,x,y)=(ax+y)(bx+x_2+y)$$ and $$\partial_{x_2}f(x_1,x_2,x,y)=(ax+y)(bx+x_1+y)$$ are products of Lorentzian polynomials.
So it suffices to show that the Hessian matrices $H_{\partial_{x}f}$ and $H_{\partial_{y}f}$ associated
to $\partial_{x}f(x_1,x_2,x,y)$ and $\partial_{y}f(x_1,x_2,x,y)$ have only one positive eigenvalue, respectively. 
 
(i) For $a=1$, we obtain the reduced form of the Hessian matrix $H_{\partial_{x}f}$ as follows by Lemma~\ref{l1.2}.
\begin{align*}
H_{\partial_{x}f}=\begin{pmatrix}
1        & 2b        & b+1      \\
\        & \        &  \      \\
4b   & 6b^2   & 2+4b+2b^2         \\
\        & \        &  \      \\
2b+2        & 2+4b+2b^2   & 4+4b             \\
\end{pmatrix}.
\end{align*}
 
 Obviously, $\mathrm{Tr}(H_{\partial_{x}f})$ is positive and
\begin{align*}
 H_{11}+H_{22}+H_{33} = -4b^4 + 8b^3 -4b^2 - 16b -2 < 0, \quad {\rm for}\ b \in \mathbb{N}^*
\end{align*}
and
\begin{align*}
 \mathrm{det}(H_{\partial_{x}f})= 4(b^2-1) \geq 0, \quad {\rm for}\ b \in \mathbb{N}^*. 
 \end{align*}
Thus the Hessian matrix $H_{\partial_{x}f}$ has exactly one positive eigenvalue by Lemma~\ref{l1.3}.

The reduced form of the Hessian matrix $H_{\partial_{y}f}$ is
\begin{align*}
H_{\partial_{y}f}=\begin{pmatrix}
1        & b+1        & 2      \\
\        & \        &  \      \\
2b+2   & 2(b+1)^2   & 4b+4         \\
\        & \        &  \      \\
4        & 4b+4   & 6            \\
\end{pmatrix}.
\end{align*}
 
 Obviously, $\mathrm{det}(H_{\partial_{y}f})=0 $, $\mathrm{Tr}(H_{\partial_{y}f})$ is positive and
\begin{align*}
 H_{11}+H_{22}+H_{33} = -4b^2-8b-6 < 0, \quad {\rm for}\ b \in \mathbb{N}^*.
\end{align*}

Thus $H_{\partial_{y}f}$ has exactly one positive eigenvalue by Lemma~\ref{l1.3}.

(ii)  For $a=b$, we obtain the reduced form of the Hessian matrix $H_{\partial_{x}f}$ as follows by Lemma~\ref{l1.2}.
\begin{align*}
H_{\partial_{x}f}=\begin{pmatrix}
b        & 2b^2        & 2b      \\
\        & \        &  \      \\
4b^2   & 6b^3   & 2b+6b^2         \\
\        & \        &  \      \\
4b        & 2b+6b^2   & 6b+2            \\
\end{pmatrix}.
\end{align*}
 
 Obviously, $\mathrm{Tr}(H_{\partial_{x}f})$ is positive and
\begin{align*}
  H_{11}+H_{22}+H_{33} = -2b^4-12b^3-6b^2+2b < 0, \quad  {\rm for}\ b \in \mathbb{N}^*
\end{align*}
and
\begin{align*}
 \mathrm{det}(H_{\partial_{x}f})= 4b^3(b-1) \geq 0, \quad   {\rm for}\ b \in \mathbb{N}^*. 
 \end{align*}
Thus $H_{\partial_{x}f}$ has exactly one positive eigenvalue by Lemma~\ref{l1.3}.

The reduced form of the Hessian matrix $H_{\partial_{y}f}$ is
\begin{align*}
H_{\partial_{y}f}=\begin{pmatrix}
1        & 2b        & 2      \\
\        & \        &  \      \\
4b   & 6b^2+2b    & 6b+2         \\
\        & \        &  \      \\
4        & 6b+2   & 6            \\
\end{pmatrix}.
\end{align*}
 
 Obviously, $\mathrm{Tr}(H_{\partial_{y}f})$ is positive and
\begin{align*}
  H_{11}+H_{22}+H_{33} = -2b^2-10b-6 < 0, \quad {\rm for}\ b \in \mathbb{N}^*
\end{align*}
and
\begin{align*}
 \mathrm{det}(H_{\partial_{y}f})= 4(b-1) \geq 0, \quad {\rm for}\ b \in \mathbb{N}^*. 
 \end{align*}
Thus $H_{\partial_{y}f}$ has exactly one positive eigenvalue by Lemma~\ref{l1.3}.
Therefore,
the multivariate polynomial $f(x_1,x_2,x,y)$ is a Lorentzian polynomial by the definition.
This completes the proof.

\end{proof}
\begin{remark}
The hypotheses of Lemma~\ref{p3.2} are sufficient but not necessary, and the exact region of validity appears to be complicated. For example, numerical experiments suggest that the Hessian of $\partial_xf$
is Lorentzian for $(a,b)=(3,2)$ and $(4,2)$,
 but not for $(2,1)$ or $(5,2)$. Thus, the conditions $a=1$ and $a=b$ are not expected to be optimal.
\end{remark}

\begin{definition}\label{d3.1}
Let $\mathcal{J}(l,m)$ be a coloured graph, obtained by adding $m$ and $l-m$ edges connecting the free vertex $x_i $ and the complete graphs $K_m$ and $K_l$, respectively (see Figure~\ref{f15}). 
\end{definition}

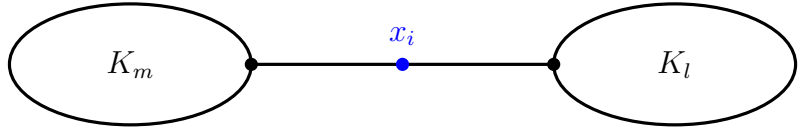
\begin{figure}[htbp] 
\centering
\begin{tikzpicture}[scale = 0.8]

\draw[line width=1.2pt](-3,0) ellipse (2cm and 1cm)node {$K_{m}$};

\draw[line width=1.2pt](6,0) ellipse (2cm and 1cm)node {$K_{l}$};

\fill(4,0) circle(0.6ex);

\fill(-1,0) circle(0.6ex);

\draw[line width=1.2pt] (-1,0) -- (4,0);

\fill[blue](1.5,0) circle(0.6ex)node [above=3pt] {$x_i$};

\end{tikzpicture}
\caption{$\mathcal{J}(l,m)$
}\label{f15}
\end{figure}

\newpage

Now we prove the pre-Lorentzian property of the coloured graph $\mathcal{F}_n(l,m,t,s)$.

\begin{thm}\label{t1.1}
Suppose that $n\in \mathbb{N}^*$, $m,s,t,l\in \mathbb{N}$, $l\geq m \geq s$ and $l \geq t+m$. 
Then we have the following.
\begin{enumerate}[(i)]
  \item If $m=1$, $m=s+1$ or $l=t+1$, then $\mathcal{F}_n(l,m,t,s)$ is a pre-Lorentzian graph.
  \item If $l\ge 1$ and $t=0$, then $\mathcal{F}_n(l,m,t,s)$ is a pre-Lorentzian graph.
\end{enumerate}

\end{thm}
\begin{proof}
Before proceeding with the proof, we note that under either condition (i) or condition (ii), we have 
$l\ge 1$.
When we select the vertices $c$ and $x_i$ successively in Lemma~\ref{l2.1} (i),
we obtain the following specific expression for the coloured independence polynomial $C( \mathcal{F}_n(l,m,t,s);x,(x_i))$.
\begin{align*}
   &C( \mathcal{F}_n(l,m,t,s);x,(x_i)) \\
   =& C( \mathcal{F}_n(l,m,t,s)- c;x,(x_i))+ x C( \mathcal{F}_n(l,m,t,s)- N[c];x,(x_i))\\
   =& I(K_{m-s};x)^n I(K_{l-t};x)^n \prod^n_{i=1}C(\mathcal{J}(l,m);x,(x_i)) + x I(K_{m};x)^{n} I(K_{m-s};x)^{n} I(K_{l-t};x)^{n}\\
   =&I(K_{m-s};x)^n I(K_{l-t};x)^n \prod^n_{i=1}(C(\mathcal{J}(l,m) - x_i;x,(x_i))+xC(\mathcal{J}(l,m) - N[x_i];x,(x_i)) )\\
    & + x I(K_{m};x)^{n} I(K_{m-s};x)^{n} I(K_{l-t};x)^{n}\\
   =& I(K_{m-s};x)^n I(K_{l-t};x)^n \Bigg [\prod^n_{i=1}(I(K_m;x) I(K_{l};x)+x_i I(K_{m};x) ) +x I(K_{m};x)^{n}\Bigg]\\
   =& I(K_{m-s};x)^n I(K_{l-t};x)^n I(K_m;x)^n \Bigg [\prod^n_{i=1} (I(K_{l};x)+x_i ) +x\Bigg]\\
   =& [(m-s)x+1]^n [(l-t)x+1]^n (mx+1)^n \Bigg [\prod^n_{i=1} (lx +x_i +1 ) +x\Bigg],
\end{align*}
where $x$ is the variable associated to the bound colour and $x_i$ are the variables associated to
 the free colours.

In order to prove that $\mathcal{F}_n(l,m,t,s)$ is a pre-Lorentzian graph, 
we need to show that there exists a $k\in\mathbb{N}$,
such that $(xy)^k H(C( \mathcal{F}_n(l,m,t,s);x,(x_i)))$ is a Lorentzian polynomial under the condition (i) or (ii).
Denote by
\begin{align}\label{recu-1}
  \overline{H}(C( \mathcal{F}_n(l,m,t,s);x,(x_i))) =  [(m-s)x+y]^n [(l-t)x+y]^n (mx+y)^n \Bigg [\prod^n_{i=1} (lx +x_i +y ) +xy^{n-1}\Bigg].
\end{align}
Obviously,
if $m=0$, $m=s$ or $l=t$,
we have
$$ \overline{H}(C( \mathcal{F}_n(l,m,t,s);x,(x_i))) =y^{jn}H(C( \mathcal{F}_n(l,m,t,s);x,(x_i))),$$
where $j\in\{0,1,2,3\}$ is the number of degenerate factors among $m, m-s$ and $l-t$.
Otherwise, we have
$$ \overline{H}(C( \mathcal{F}_n(l,m,t,s);x,(x_i))) =H(C( \mathcal{F}_n(l,m,t,s);x,(x_i))).$$
As stated by Bendjeddou and Hardiman~\cite[Remark 3.13]{BH2025},
it suffices to prove that there exists a $k\in\mathbb{N}$,
such that $(xy)^k  \overline{H}(C( \mathcal{F}_n(l,m,t,s);x,(x_i)))$ is a Lorentzian polynomial under the condition (i) or (ii).

For $n=1$, since 
$$ \overline{H}(C( \mathcal{F}_1(l,m,t,s);x,(x_i)))= [(m-s)x+y] [(l-t)x+y] (mx+y) [(l+1)x +x_1 +y ]$$ is a product of Lorentzian polynomials, $ \overline{H}(C( \mathcal{F}_1(l,m,t,s);x,(x_i)))$ is a Lorentzian polynomial by Proposition~\ref{p1.2}.

For $n=2$, we have
$$ \overline{H}(C( \mathcal{F}_2(l,m,t,s);x,(x_i)))=[(m-s)x+y]^2 [(l-t)x+y]^2 (mx+y)^2  [ (lx +x_1 +y ) (lx +x_2 +y ) +xy].$$ 
By Lemma~\ref{p3.2},
we obtain that 
\begin{itemize}
\item if $m=1$,
then the multivariate polynomial $$(mx+y) [ (lx +x_1 +y ) (lx +x_2 +y ) +xy]$$ is a Lorentzian polynomial.
\item if $m=s+1$,
then the multivariate polynomial $$[(m-s)x+y] [ (lx +x_1 +y ) (lx +x_2 +y ) +xy]$$ is a Lorentzian polynomial.
\item if $l=t+1$,
then the multivariate polynomial $$[(l-t)x+y] [ (lx +x_1 +y ) (lx +x_2 +y ) +xy]$$ is a Lorentzian polynomial.
\item if $l\ge 1$ and $t=0$,
then the multivariate polynomial $$[(l-t)x+y] [ (lx +x_1 +y ) (lx +x_2 +y ) +xy]$$ is a Lorentzian polynomial.
\end{itemize}
Thus $ \overline{H}(C( \mathcal{F}_2(l,m,t,s);x,(x_i)))$ is a Lorentzian polynomial by Proposition~\ref{p1.2}.

For $n>2$, since $[(m-s)x+1]^n [(l-t)x+1]^n (mx+1)^n$ is  a Lorentzian polynomial, it suffices to show that

\begin{align*}
 h(x,y, (x_i)) =  xy^{n-1} + \prod^n_{i=1}(lx + x_i + y )
\end{align*}
 is  a Lorentzian polynomial after being multiplied by $(xy)^k$. 
By the definition of Lorentzian polynomials,
we first prove that Supp$( h(x,y, (x_i)))$ is M-convex. 
Note that
\begin{align*}
 \mathrm{Supp}(h(x,y, (x_i)))
 =& \mathrm{Supp}(xy^{n-1} + \prod^n_{i=1}( lx +x_i + y ))\\
 =& \mathrm{Supp}( \prod^n_{i=1}( lx +x_i + y )).
\end{align*}
Obviously, $ \prod^n_{i=1}(x_i + x + y )$ is a product of Lorentzian polynomials.
Thus Supp$(h(x,y, (x_i)))$ is M-convex.

The remainder of the proof is to show that,
for all sufficiently large $k$, 
every $(n+2k-2)$th partial derivative of
\begin{align}\label{e1.2}
   p(x,y, (x_i)):= (xy)^k \left[ xy^{n-1} + \prod^n_{i=1}( lx +x_i + y ) \right]
\end{align}
has a Hessian matrix with at most one positive eigenvalue. 
For any $j \in [n]$, we have
\begin{align*}
  \partial_{x_j} p(x,y, (x_i))=(xy)^k  \prod^n_{i \neq j}(lx +x_i +  y ).
\end{align*}
The right-hand side is a product of Lorentzian polynomials.
Hence $ \partial_{x_j} p(x,y, (x_i))$ is  Lorentzian. 
Next it suffices to consider that the partial derivatives of the form
\begin{align*}
  \partial_{x}^{q}  \partial_{y}^{n+2k-q-2} p(x,y, (x_i))
\end{align*}
have the Hessian matrix with at most one positive eigenvalue for sufficiently large $k$. 
We proceed by considering the following four cases based on the value of $q$.

(1) For $q < k-1 $ or $q > k+1 $, 
this case will eliminate the terms $x^{k+1}y^{n+k-1}$ in $p(x,y, (x_i))$, 
the remaining terms are of the form $(xy)^k  \prod^n_{i =1}( lx + x_i +y )$, 
which is a product of Lorentzian polynomials. 
Thus $\partial_{x}^{q}  \partial_{y}^{n+2k-q-2} p(x,y, (x_i))$ is a Lorentzian polynomial.

(2) For $q=k-1$, 
in this case,
the remaining terms are those for which the degree of $x$ is at least $k-1$,
and the degree of $y$ is at least $n+k-1$.
There are the following six terms
\begin{align*}
   x^{k-1}y^{n+k-1}x_i x_j,\quad  x^{k}y^{n+k-1}x_i,\quad  x^{k-1}y^{n+k}x_i,\quad
   x^{k+1}y^{n+k-1},\quad   x^{k-1}y^{n+k+1},\ \mathrm{and}  \  x^{k}y^{n+k}.
\end{align*}
And their coefficients are
 \begin{align*}
   0,\quad 1,\quad  0,\quad
   ln+1,\quad   0,\ \mathrm{and}  \ 1,
\end{align*}
respectively.

After differentiating, we have
 \begin{align*}
   \partial_{x}^{k-1}  \partial_{y}^{n+k-1} p(x,y, (x_i))= (k-1)!(n+k-1)!\bigg (\sum_{i =1}^n k x_i x + (ln+1)(k+1)k\frac{x^2}{2} + k(n+k)xy \bigg ).
 \end{align*}

By removing the common factor $(k-1)!(n+k-1)!$,
we have the following reduced form of the Hessian matrix by Lemma~\ref{l1.2}.

 \begin{align*}
H =\begin{pmatrix}
0      & k        & 0     \\
                 &                             &                   \\
nk         & (ln+1)(k+1)k        & k(n+k)      \\
                 &                             &                   \\
0          & k(n+k)        & 0     \\
\end{pmatrix}.
\end{align*}

Obviously, $\mathrm{det}(H)=0$, $\mathrm{Tr}(H)$ is positive and
\begin{align*}
  H_{11}+H_{22}+H_{33} = -k^2(n+k)^2 - nk^2 < 0, \quad \mathrm{for}  \ n >2.
\end{align*}
Thus $H$ has exactly one positive eigenvalue by Lemma \ref{l1.3}.

(3) For $q=k$, 
in this case,
the remaining terms are those for which the degree of $x$ is at least $k$,
and the degree of $y$ is at least $n+k-2$.
There are the following six terms
\begin{align*}
  x^{k}y^{n+k-2}x_ix_j,\quad  x^{k+1}y^{n+k-2}x_i,\quad  x^{k}y^{n+k-1}x_i,\quad
   x^{k+2}y^{n+k-2},\quad   x^{k}y^{n+k},\ \mathrm{and}  \  x^{k+1}y^{n+k-1}.
\end{align*}

And their coefficients are
 \begin{align*}
  1,\quad  l(n-1),\quad  1,\quad
  l^2 \binom{n}{2},\quad   1,\ \mathrm{and}  \ l n+1,
\end{align*}
respectively.

After differentiating, we have
 \begin{align*}
   &\partial_{x}^{k}  \partial_{y}^{n+k-2} p(x,y, (x_i)) \\
   =& k!(n+k-2)! \bigg (\sum_{i \neq j}x_ix_j +\sum_{i=1}^{n}l(n-1)(k+1)x_ix
  + l^2 \binom{n}{2}(k+2)(k+1)\frac{x^2}{2} \\
  & +\sum_{i=1}^{n}(n+k-1)x_iy  +(n+k)(n+k-1)\frac{y^2}{2}  + (ln+1)(k+1)(n+k-1)xy \bigg ).
 \end{align*}
By removing the common factor $k!(n+k-2)!$,
we have the following reduced form of the Hessian matrix by Lemma~\ref{l1.2}.
 \begin{align*}
H =\begin{pmatrix}
n-1      & l(n-1)(k+1)        & n+k-1     \\
                 &                             &                   \\
ln(n-1)(k+1)         & l^2 \binom{n}{2}(k+2)(k+1)       & (ln+1)(k+1)(n+k-1)      \\
                 &                             &                   \\
n(n+k-1)          & (ln+1)(k+1)(n+k-1)       & (n+k)(n+k-1)     \\
\end{pmatrix}.
\end{align*}

Obviously, $\mathrm{Tr}(H)$ is positive. 
Without loss of generality, we assume $k$ to be sufficiently large. 
By Lemma~\ref{l1.3},
it suffices to prove $H_{11}+H_{22}+H_{33}\le 0$ and $\mathrm{det}(H)\ge 0$ for sufficiently large $k$.
Since $H_{11}$ has order $k^4$, whereas both $H_{22}$ and $H_{33}$ have order $k^2$, 
it suffices to prove $H_{11} < 0$. 
By the direct calculation,
we have
\begin{align*}
  H_{11} = k^4 \bigg (l^2 \binom{n}{2}-(ln+1)^2 \bigg )+O(k^3) < 0,\quad \mathrm{for}  \ n>2.
\end{align*}

Let $H'$ be the matrix obtained from the matrix $H$ by removing all terms containing $k$.
\begin{align*}
H' =\begin{pmatrix}
n-1      & l(n-1)       & 1     \\
                 &                             &                   \\
ln(n-1) & l^2 \binom{n}{2}       & ln+1     \\
                 &                             &                   \\
n         & ln+1       & 1    \\
\end{pmatrix}.
\end{align*}

Expanding $\mathrm{det}(H)$, we obtain
\begin{align*}
  \mathrm{det}(H) = k^4 \cdot \mathrm{det}(H') + O(k^3),
\end{align*}
which implies that it suffices to show that $\mathrm{det}(H') >0$. 
By the direct calculation, we have
\begin{align*}
  \mathrm{det}(H')=\frac{l^2 n^2}{2} - \left(\frac{l^2}{2} + 1\right)n + 1 > 0,\quad \mathrm{for}  \ n > 2.
\end{align*}

Thus $H$ has exactly one positive eigenvalue by Lemma~\ref{l1.3}.

(4) For $q=k+1$, 
in this case,
the remaining terms are those for which the degree of $x$ is at least $k+1$,
and the degree of $y$ is at least $n+k-3$.
There are the following six terms
\begin{align*}
  x^{k+1}y^{n+k-3}x_ix_j,\quad  x^{k+2}y^{n+k-3}x_i,\quad  x^{k+1}y^{n+k-2}x_i,\quad
   x^{k+3}y^{n+k-3},\quad   x^{k+1}y^{n+k-1},\ \mathrm{and}  \  x^{k+2}y^{n+k-2}.
\end{align*}

And their coefficients are
 \begin{align*}
  l(n-2),\quad l^2 \binom{n-1}{2},\quad  l(n-1),\quad
  l^3 \binom{n}{3},\quad   ln+1,\ \mathrm{and}  \ l^2 \binom{n}{2},
\end{align*}
respectively.

After differentiating, we have
 \begin{align*}
   &\partial_{x}^{k+1}  \partial_{y}^{n+k-3} p(x,y, (x_i)) \\
   =& (k+1)!(n+k-3)! \bigg(\sum_{i \neq j}l(n-2)x_ix_j +\sum_{i=1}^{n}l^2\binom{n-1}{2}(k+2)x_ix
  + l^3\binom{n}{3}(k+3) \\
  & \times (k+2)\frac{x^2}{2}+\sum_{i=1}^{n}l(n-1)(n+k-2)x_iy  +(ln+1)(n+k-1)(n+k-2)\frac{y^2}{2} \\
    &+l^2 \binom{n}{2}(k+2)(n+k-2)xy \bigg ).
 \end{align*}
By removing the common factor $(k+1)!(n+k-3)!$,
we have the following reduced form of the Hessian matrix by Lemma~\ref{l1.2}.
 \begin{align*}
H =\begin{pmatrix}
l(n-1)(n-2)      & l^2\binom{n-1}{2}(k+2)        & l(n-1)(n+k-2)     \\
                 &                             &                   \\
l^2n\binom{n-1}{2}(k+2)      &  l^3\binom{n}{3}(k+3)(k+2)       & l^2\binom{n}{2}(k+2)(n+k-2)      \\
                 &                             &                   \\
ln(n-1)(n+k-2)          & l^2\binom{n}{2}(k+2)(n+k-2)       &  (ln + 1)(n + k -1)(n + k - 2)     \\
\end{pmatrix}.
\end{align*}

Obviously, $\mathrm{Tr}(H)$ is positive. 
Since $H_{11}$ has order $k^4$, whereas both $H_{22}$ and $H_{33}$ have order $k^2$, 
it suffices to prove $H_{11} < 0$. 
By the direct calculation,
we have
\begin{align*}
 H_{11} = k^4 \bigg (l^3\binom{n}{3}(ln + 1) - l^4\binom{n}{2}^2 \bigg ) + O(k^3) < 0,\quad n>2.
\end{align*}

Let $H'$ be the matrix obtained from the matrix $H$ by removing all terms containing $k$.
\begin{align*}
H' =\begin{pmatrix}
l(n - 1)(n - 2)      & l^2\binom{n-1}{2}       & l(n-1)     \\
                 &                             &                   \\
l^2n\binom{n-1}{2}         & l^3 \binom{n}{3}       & l^2\binom{n}{2}     \\
                 &                             &                   \\
ln(n-1)         & l^2\binom{n}{2}       & ln + 1    \\
\end{pmatrix}.
\end{align*}

Expanding $\mathrm{det}(H)$, we obtain
\begin{align*}
  \mathrm{det}(H) = k^4 \cdot \mathrm{det}(H') + O(k^3),
\end{align*}
which implies that it suffices to show $\mathrm{det}(H') >0$. 
By the direct calculation, we have
\begin{align*}
  \mathrm{det}(H')= \frac{l^4 n}{12} [(l-1)n^4+(6-4l)n^3 +(5l-13)n^2 + (12-2l)n -4] >0,\quad n > 2.
\end{align*}

Thus $H$ has exactly one positive eigenvalue by Lemma~\ref{l1.3}. 
Therefore $\mathcal{F}_n(l,m,t,s)$ is a pre-Lorentzian graph under the condition (i) or (ii).

\end{proof}

The following result follows immediately from Lemma~\ref{l1.4}.

\begin{coro}\label{c1.1}
Suppose that $n \in \mathbb{N}^*$, $m,s,t,l \in \mathbb{N}$, $l\geq m \geq s$ and $l \geq t+m$. 
Then we have the following.
\begin{enumerate}[(i)]
  \item If $m=1$, $m=s+1$ or $l=t+1$, then the independence polynomial of $\mathcal{F}_n(l,m,t,s)$ is log-concave.
  \item If  $l\ge1$ and $t=0$, then the independence polynomial of $\mathcal{F}_n(l,m,t,s)$ is log-concave.
\end{enumerate}

\end{coro}

\begin{remark}\label{r1.3}
By Lemma~\ref{p3.2},
conditions (i) and (ii) of Theorem~\ref{t1.1} and Corollary~\ref{c1.1} are only used in the case $n=2$.
For $n\neq 2$, the proof uses nothing beyond $l\ge 1$.
\end{remark}

Then we give the definition of the graph $\mathcal{G}_4(l,m,t,s)$,
which will replace an edge in the operator $E_{\mathcal{G}_4(l,m,t,s)}$.

\begin{definition}\label{d1.1}
Given two copies of $\mathcal{F}(l,m,t,s)$ with the vertices $x_1$ and $x_2$ assigned the bound colour,
define the graph $\mathcal{G}_4(l,m,t,s)$ by adding an edge that connects $x_1$ and $x_2$ (see Figure~\ref{f1.2.2}).
\end{definition}

In particular, $\mathcal{G}_4(1,0,0,0)=W_4$ and $\mathcal{G}_4(2,1,0,0)$ are trees (see Figures~\ref{f1.3} and~\ref{fig-sp-pi}, respectively).

\begin{definition}
Denote by $E_{\mathcal{G}_4(l,m,t,s)}$ the operator on graphs which replaces every edge $e$ of a simple finite undirected graph with the graph $\mathcal{G}_4(l,m,t,s)$,
where the red vertices $c$ are identified with the vertices of $e$ (we call these vertices the {\it endpoints}) (see Figures~\ref{f5.1} and~\ref{f5} for $E_{\mathcal{G}_4(1,0,0,0)}$ and $E_{\mathcal{G}_4(2,1,0,0)}$, respectively).
\end{definition}

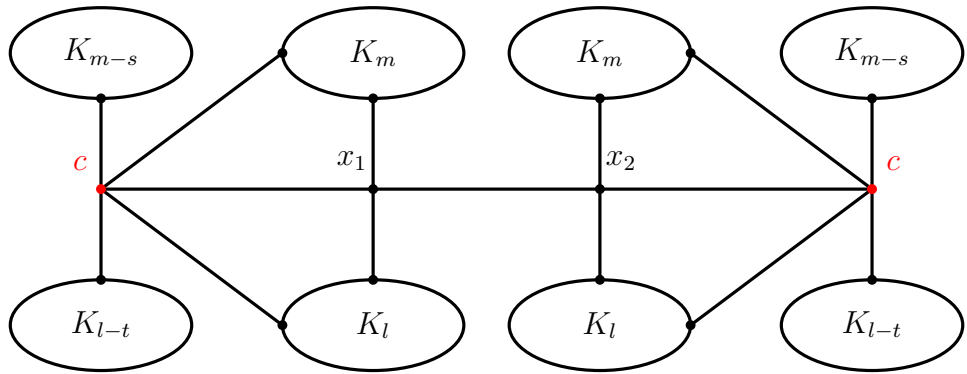
\begin{figure}[htbp] 
\centering
\begin{tikzpicture}[scale = 0.6]
\draw[line width=1.2pt](-3,0) ellipse (2cm and 1cm) node {$K_{m-s}$};
\draw[line width=1.2pt](3,0) ellipse (2cm and 1cm)node {$K_{m}$};
\draw[line width=1.2pt](-3,-6) ellipse (2cm and 1cm)node {$K_{l-t}$};
\draw[line width=1.2pt](3,-6) ellipse (2cm and 1cm)node {$K_{l}$};

\fill(3,-3) circle(0.6ex)node[left=8pt,above=3pt]{$x_1$};
\fill(-3,-1) circle(0.6ex);
\fill(3,-1) circle(0.6ex);
\fill(-3,-5) circle(0.6ex);
\fill(3,-5) circle(0.6ex);
\fill(1,0) circle(0.6ex);
\fill(1,-6) circle(0.6ex);

\draw[line width=1.2pt] (-3,-1) -- (-3,-3);
\draw[line width=1.2pt] (-3,-5) -- (-3,-3);
\draw[line width=1.2pt] (3,-3) -- (-3,-3);
\draw[line width=1.2pt] (3,-3) -- (3,-1);
\draw[line width=1.2pt] (3,-3) -- (3,-5);
\draw[line width=1.2pt] (-3,-3) -- (1,0);
\draw[line width=1.2pt] (-3,-3) -- (1,-6);

\fill[red](-3,-3) circle(0.6ex)node[left=8pt,above=3pt]{$c$};

\draw[line width=1.2pt] (8,-3) -- (3,-3);   

\draw[line width=1.2pt](8,0) ellipse (2cm and 1cm) node {$K_{m}$};
\draw[line width=1.2pt](14,0) ellipse (2cm and 1cm)node {$K_{m-s}$};
\draw[line width=1.2pt](8,-6) ellipse (2cm and 1cm)node {$K_{l}$};
\draw[line width=1.2pt](14,-6) ellipse (2cm and 1cm)node {$K_{l-t}$};

\fill(8,-3) circle(0.6ex) node [right=8pt,above=3pt]{$x_2$};

\fill(8,-1) circle(0.6ex);
\fill(14,-1) circle(0.6ex);
\fill(8,-5) circle(0.6ex);
\fill(14,-5) circle(0.6ex);
\fill(10,0) circle(0.6ex);
\fill(10,-6) circle(0.6ex);

\draw[line width=1.2pt] (8,-1) -- (8,-3);
\draw[line width=1.2pt] (8,-5) -- (8,-3);
\draw[line width=1.2pt] (14,-3) -- (8,-3);
\draw[line width=1.2pt] (14,-3) -- (14,-1);
\draw[line width=1.2pt] (14,-3) -- (14,-5);
\draw[line width=1.2pt] (14,-3) -- (10,0);
\draw[line width=1.2pt] (14,-3) -- (10,-6);

\fill[red](14,-3) circle(0.6ex)node[right=8pt,above=3pt]{$c$};

\end{tikzpicture}
\caption{$\mathcal{G}_4(l,m,t,s)$
}\label{f1.2.2}
\end{figure}

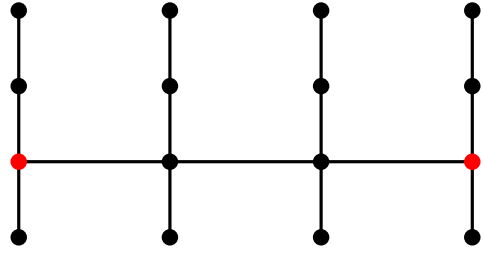
\begin{figure}[htbp] 
\centering
\begin{tikzpicture}[scale = 1]

\fill(-1,0) circle(0.6ex);
\fill(1,0) circle(0.6ex);

\draw[line width=1.2pt] (-3,0) -- (-1,0) -- (1,0) -- (3,0);
\fill(-3,-1) circle(0.6ex);
\fill(-1,-1) circle(0.6ex);
\fill(1,-1) circle(0.6ex);
\fill(3,-1) circle(0.6ex);
\draw[line width=1.2pt] (-3,0) -- (-3,-1);
\draw[line width=1.2pt] (-1,0) -- (-1,-1);
\draw[line width=1.2pt] (1,0) -- (1,-1);
\draw[line width=1.2pt] (3,0) -- (3,-1);

\fill(-3,1) circle(0.6ex);
\fill(-1,1) circle(0.6ex);
\fill(1,1) circle(0.6ex);
\fill(3,1) circle(0.6ex);
\draw[line width=1.2pt] (-3,0) -- (-3,1);
\draw[line width=1.2pt] (-1,0) -- (-1,1);
\draw[line width=1.2pt] (1,0) -- (1,1);
\draw[line width=1.2pt] (3,0) -- (3,1);

\fill(-3,2) circle(0.6ex);
\fill(-1,2) circle(0.6ex);
\fill(1,2) circle(0.6ex);
\fill(3,2) circle(0.6ex);
\draw[line width=1.2pt] (-3,1) -- (-3,2);
\draw[line width=1.2pt] (-1,1) -- (-1,2);
\draw[line width=1.2pt] (1,1) -- (1,2);
\draw[line width=1.2pt] (3,1) -- (3,2);

\fill[red](-3,0) circle(0.6ex);
\fill[red](3,0) circle(0.6ex);
\end{tikzpicture}
\caption{$\mathcal{G}_4(2,1,0,0)$
}\label{fig-sp-pi}
\end{figure}

The crucial point of our result is that the graph in $E_{\mathcal{G}_4(l,m,t,s)}(\mathbf{G})$ can also be constructed in another way.

\begin{prop}\label{p1.3}
Let $\mathbf{G}$ denote the set of all simple finite undirected graphs.
Then every graph in $E_{\mathcal{G}_4(l,m,t,s)}(\mathbf{G})$ can be constructed by N-gluing copies of $\mathcal{F}_n(l,m,t,s)$ across free vertices.
\end{prop}

\begin{proof}
Let $G = (V, E)$ be a simple finite undirected graph. For any vertex $v_i \in V(G)$, let $\mathcal{F}^i_{\mathrm{deg}(v_i)}(l,m,t,s)$
 be a copy of $\mathcal{F}_{\mathrm{deg}(v_i)}(l,m,t,s)$ and $c^i$ be its center. 
 
 For each edge $ e_{ij} = v_i v_j \in E(G)$, we N-glue the corresponding free vertices of $\mathcal{F}^i_{\mathrm{deg}(v_i)}(l,m,t,s)$ and $\mathcal{F}^j_{\mathrm{deg}(v_j)}(l,m,t,s)$. 
Note that the graph obtained in this way is exactly $E_{\mathcal{G}_4(l,m,t,s)}(G)$.

For $\{v_i,v_j\}\in E(G)$,
let $e_{ij}$ be the unique edge connecting the free vertices of $\mathcal{F}^i_{\mathrm{deg}(v_i)}(l,m,t,s)$ and $\mathcal{F}^j_{\mathrm{deg}(v_j)}(l,m,t,s)$. 
It is clear that we can choose a unique $\mathcal{G}_4(l,m,t,s)$, which contains $e_{ij}$ with two endpoints $c^i$ and $c^j$.
This $\mathcal{G}_4(l,m,t,s)$ can be seen as the image of the edge $\{v_i,v_j\}$ under $E_{\mathcal{G}_4(l,m,t,s)}$.

\end{proof}

 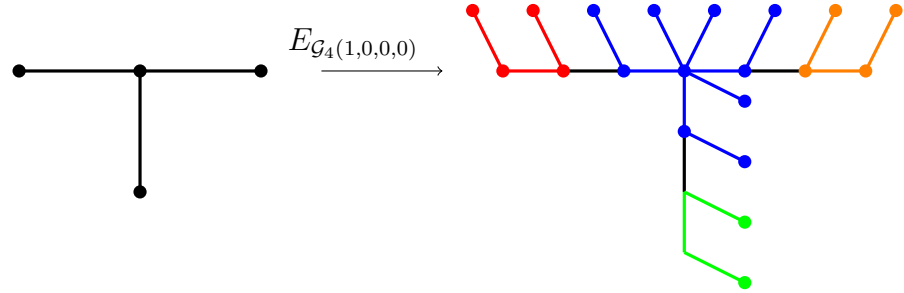
\begin{figure}[htbp]
\centering

\begin{tikzpicture}[scale = 0.8]
\fill(-6,0) circle(0.6ex);
\fill(-4,0) circle(0.6ex);
\fill(-2,0) circle(0.6ex);
\fill(-4,-2) circle(0.6ex);

\draw[line width=1.2pt] (-6,0) -- (-4,0) -- (-2,0);
\draw[line width=1.2pt] (-4,0) -- (-4,-2);
\draw[line width=1.2pt] (5,-1) -- (5,-2);
\draw[->] (-1,0) -- (1,0) node[above=10pt,left=5pt] {$E_{\mathcal{G}_4(1,0,0,0)}$};

\draw[line width=1.2pt, red] (2,0) -- (3,0);
\draw[line width=1.2pt] (3,0) -- (4,0);
\draw[line width=1.2pt, blue] (4,0) -- (5,0) -- (6,0);
\draw[line width=1.2pt] (6,0) -- (7,0);
\draw[line width=1.2pt, orange] (7,0) -- (8,0);

\draw[line width=1.2pt, red] (2,0) -- (1.5,1);
\draw[line width=1.2pt, red] (3,0) -- (2.5,1);

\draw[line width=1.2pt, blue] (4,0) -- (3.5,1);
\draw[line width=1.2pt, blue] (6,0) -- (6.5,1);
\draw[line width=1.2pt, blue] (5,0) -- (4.5,1);
\draw[line width=1.2pt, blue] (5,0) -- (5.5,1);
\draw[line width=1.2pt, orange] (7,0) -- (7.5,1);
\draw[line width=1.2pt, orange] (8,0) -- (8.5,1);

\fill[red](2,0) circle(0.6ex);
\fill[red](3,0) circle(0.6ex);
\fill[blue](4,0) circle(0.6ex);
\fill[blue](5,0) circle(0.6ex);
\fill[blue](6,0) circle(0.6ex);
\fill[orange](7,0) circle(0.6ex);
\fill[orange](8,0) circle(0.6ex);

\fill[red](1.5,1) circle(0.6ex);
\fill[red](2.5,1) circle(0.6ex);
\fill[blue](3.5,1) circle(0.6ex);
\fill[blue](4.5,1) circle(0.6ex);
\fill[blue](5.5,1) circle(0.6ex);
\fill[blue](6.5,1) circle(0.6ex);
\fill[orange](7.5,1) circle(0.6ex);
\fill[orange](8.5,1) circle(0.6ex);


\draw[line width=1.2pt, blue] (5,0) -- (5,-1);
\draw[line width=1.2pt, blue] (5,0) -- (6,-0.5);
\draw[line width=1.2pt, blue] (5,-1) -- (6,-1.5);

\fill[blue](5,-1) circle(0.6ex);
\fill[blue](6,-1.5) circle(0.6ex);
\fill[blue](6,-0.5) circle(0.6ex);

\draw[line width=1.2pt, green] (5,-2) -- (6,-2.5);
\draw[line width=1.2pt, green] (5,-3) -- (6,-3.5);
\draw[line width=1.2pt, green] (5,-2) -- (5,-3);

\fill[green](6,-3.5) circle(0.6ex);
\fill[green](6,-2.5) circle(0.6ex);

\end{tikzpicture}
\caption{N-gluing copies of $\mathcal{F}(1,0,0,0)$}\label{f5.1}
\end{figure}

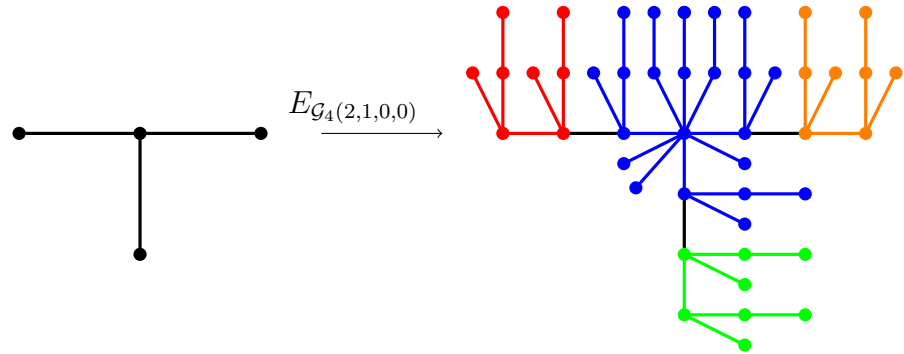
\begin{figure}[htbp]
\centering

\begin{tikzpicture}[scale = 0.8]
\fill(-6,0) circle(0.6ex);
\fill(-4,0) circle(0.6ex);
\fill(-2,0) circle(0.6ex);
\fill(-4,-2) circle(0.6ex);

\draw[line width=1.2pt] (-6,0) -- (-4,0) -- (-2,0);
\draw[line width=1.2pt] (-4,0) -- (-4,-2);
\draw[line width=1.2pt] (5,-1) -- (5,-2);
\draw[->] (-1,0) -- (1,0) node[above=10pt,left=5pt] {$E_{\mathcal{G}_4(2,1,0,0)}$};

\draw[line width=1.2pt, red] (2,0) -- (3,0);
\draw[line width=1.2pt] (3,0) -- (4,0);
\draw[line width=1.2pt, blue] (4,0) -- (5,0) -- (6,0);
\draw[line width=1.2pt] (6,0) -- (7,0);
\draw[line width=1.2pt, orange] (7,0) -- (8,0);

\draw[line width=1.2pt, red] (2,0) -- (2,1) -- (2,2);
\draw[line width=1.2pt, red] (2,0) -- (1.5,1);
\draw[line width=1.2pt, red] (3,0) -- (3,1) -- (3,2);
\draw[line width=1.2pt, red] (3,0) -- (2.5,1);

\draw[line width=1.2pt, blue] (4,0) -- (4,1) -- (4,2);
\draw[line width=1.2pt, blue] (4,0) -- (3.5,1);
\draw[line width=1.2pt, blue] (6,0) -- (6,1) -- (6,2);
\draw[line width=1.2pt, blue] (6,0) -- (6.5,1);
\draw[line width=1.2pt, blue] (5,0) -- (4.5,1) -- (4.5,2);
\draw[line width=1.2pt, blue] (5,0) -- (5,1) -- (5,2);
\draw[line width=1.2pt, blue] (5,0) -- (5.5,1) -- (5.5,2);
\draw[line width=1.2pt, orange] (7,0) -- (7,1) -- (7,2);
\draw[line width=1.2pt, orange] (7,0) -- (7.5,1);
\draw[line width=1.2pt, orange] (8,0) -- (8,1) -- (8,2);
\draw[line width=1.2pt, orange] (8,0) -- (8.5,1);

\fill[red](2,0) circle(0.6ex);
\fill[red](3,0) circle(0.6ex);
\fill[blue](4,0) circle(0.6ex);
\fill[blue](5,0) circle(0.6ex);
\fill[blue](6,0) circle(0.6ex);
\fill[orange](7,0) circle(0.6ex);
\fill[orange](8,0) circle(0.6ex);

\fill[red](2,1) circle(0.6ex);
\fill[red](3,1) circle(0.6ex);
\fill[red](1.5,1) circle(0.6ex);
\fill[red](2.5,1) circle(0.6ex);
\fill[blue](4,1) circle(0.6ex);
\fill[blue](3.5,1) circle(0.6ex);
\fill[blue](4.5,1) circle(0.6ex);
\fill[blue](5,1) circle(0.6ex);
\fill[blue](5.5,1) circle(0.6ex);
\fill[blue](6,1) circle(0.6ex);
\fill[blue](6.5,1) circle(0.6ex);
\fill[orange](7,1) circle(0.6ex);
\fill[orange](8,1) circle(0.6ex);
\fill[orange](7.5,1) circle(0.6ex);
\fill[orange](8.5,1) circle(0.6ex);

\fill[red](2,2) circle(0.6ex);
\fill[red](3,2) circle(0.6ex);
\fill[blue](4,2) circle(0.6ex);
\fill[blue](5,2) circle(0.6ex);
\fill[blue](5,2) circle(0.6ex);
\fill[blue](5.5,2) circle(0.6ex);
\fill[blue](4.5,2) circle(0.6ex);
\fill[blue](6,2) circle(0.6ex);
\fill[orange](7,2) circle(0.6ex);
\fill[orange](8,2) circle(0.6ex);

\draw[line width=1.2pt, blue] (5,0) -- (5,-1);
\draw[line width=1.2pt, blue] (5,0) -- (4,-0.5);
\draw[line width=1.2pt, blue] (5,0) -- (4.2,-0.9);
\draw[line width=1.2pt, blue] (5,0) -- (6,-0.5);
\draw[line width=1.2pt, blue] (5,-1) -- (6,-1) -- (7,-1);
\draw[line width=1.2pt, blue] (5,-1) -- (6,-1.5);

\fill[blue](5,-1) circle(0.6ex);
\fill[blue](4.2,-0.9) circle(0.6ex);
\fill[blue](6,-1.5) circle(0.6ex);
\fill[blue](6,-1) circle(0.6ex);
\fill[blue](7,-1) circle(0.6ex);
\fill[blue](4,-0.5) circle(0.6ex);
\fill[blue](6,-0.5) circle(0.6ex);

\draw[line width=1.2pt, green] (5,-2) -- (6,-2) -- (7,-2);
\draw[line width=1.2pt, green] (5,-2) -- (6,-2.5);
\draw[line width=1.2pt, green] (5,-3) -- (6,-3) -- (7,-3);
\draw[line width=1.2pt, green] (5,-3) -- (6,-3.5);
\draw[line width=1.2pt, green] (5,-2) -- (5,-3);

\fill[green](5,-2) circle(0.6ex);
\fill[green](6,-2) circle(0.6ex);
\fill[green](7,-2) circle(0.6ex);
\fill[green](5,-3) circle(0.6ex);
\fill[green](6,-3) circle(0.6ex);
\fill[green](7,-3) circle(0.6ex);
\fill[green](6,-3.5) circle(0.6ex);
\fill[green](6,-2.5) circle(0.6ex);

\end{tikzpicture}
\caption{N-gluing copies of $\mathcal{F}(2,1,0,0)$}\label{f5}
\end{figure}

Now we are in a position to present the main results of this paper.
\begin{thm}\label{t1.2}
Suppose that $l \in \mathbb{N}^*$, $m,s,t \in \mathbb{N}$. 
Let $\mathbf{G}$ denote the set of all simple finite undirected graphs (not necessarily connected) with no vertex of degree $2$. 
Then  the
independence polynomials of all graphs in $E_{\mathcal{G}_4(l,m,t,s)}(\mathbf{G})$ are log-concave.
\end{thm}

\begin{proof}
From the definition,
the coloured independence polynomials satisfy
$$C(\mathcal{G}_1 \sqcup \mathcal{G}_2;x_{i(v)},y_{j(v)}) = C(\mathcal{G}_1;x_{i(v)}) \cdot C(\mathcal{G}_2;y_{j(v)}),$$
for any two coloured graphs $\mathcal{G}_1=(G_1,i)$ and $\mathcal{G}_2=(G_2,j)$.
By Propositions~\ref{p1.1} and~\ref{p1.2},
it is obvious that the disjoint union of graphs preserves the pre-Lorentzian property.
So it suffices to discuss the connected graphs.

Let $n \in \mathbb{N}^*$ and $G$ be a connected graph in $E_{\mathcal{G}_4(l,m,t,s)}(\mathbf{G})$. 
By  Proposition~\ref{p1.3}, we know that $G$ can be constructed by N-gluing copies of $\mathcal{F}_n(l,m,t,s)$ across free vertices. 
By Remark~\ref{r1.3},
we have that the coloured graph $\mathcal{F}_n(l,m,t,s)$ is a pre-Lorentzian graph.
So the graph $G$ is also a pre-Lorentzian graph by Lemma~\ref{l1.1}.
If we consider $G$ as a partitioned graph without free vertices,
then the independence polynomial of $G$ is log-concave by Propositions~\ref{p1.1},~\ref{p1.2} and Lemma~\ref{l1.4}.
\end{proof}

If $\mathbf{G}$ denotes the set of all simple finite undirected graphs (not necessarily connected),
then the conditions (i) and (ii) of Theorem~\ref{t1.1} are necessary.
So we obtain the following result,
whose proof is similar to Theorem~\ref{t1.2}.
We omit it for brevity.
\begin{thm}\label{t1.3}
Suppose that  $m,s,t,l \in \mathbb{N}$, $l\geq m \geq s$ and $l \geq t+m$. 
Then we have the following results.
 \begin{enumerate}[(i)]
  \item If $m=1$, $m=s+1$ or $l=t+1$, then the
independence polynomials of all graphs in $E_{\mathcal{G}_4(l,m,t,s)}(\mathbf{G})$ are log-concave.
  \item If $l\ge 1$ and $t=0$, then the
independence polynomials of all graphs in $E_{\mathcal{G}_4(l,m,t,s)}(\mathbf{G})$ are log-concave.
\end{enumerate}
\end{thm}

The following result,
which is the main theorem of~\cite{BH2025},
follows immediately from Theorem~\ref{t1.3}.

\begin{coro}[\rm\cite{BH2025}]
The independence polynomials of all graphs in $E_{\mathcal{G}_4(1,0,0,0)}(\mathbf{G})$ are log-concave.
\end{coro}

\begin{coro}
The independence polynomials of all graphs in $E_{\mathcal{G}_4(2,1,0,0)}(\mathbf{G})$ are log-concave.
\end{coro}

If $\mathbf{G}$ is the set of all simple finite forests,
then the graphs in the following two special cases $E_{\mathcal{G}_4(1,0,0,0)}(\mathbf{G})$ and $E_{\mathcal{G}_4(2,1,0,0)}(\mathbf{G})$ are clawed trees.
So our results not only make some progress on Alavi et al.'s conjecture,
but also complement the result of Chudnovsky and Seymour.

\section*{Acknowledgements}
\hspace*{\parindent}

This work was supported partially by the National Natural Science Foundation of China (No. 12371330).

The authors thank the anonymous reviewers for their  careful reading and valuable suggestions.


\end{document}